\documentclass[12pt,a4paper]{amsart}
\usepackage{}

\makeatletter
\tagsleft@false
\makeatother

\usepackage{mathrsfs}
\usepackage{tikz}
\usepackage{amsfonts}
\usepackage{amsmath}
\usepackage{amssymb}
\usepackage{enumerate}
\usepackage{hyperref}
\usepackage{latexsym}
\usepackage{xcolor}
\usepackage[shortlabels]{enumitem}
\usepackage[numbers,sort&compress]
{natbib}
\setlist[enumerate,1]{label={\upshape(\roman*)}}

 \makeatletter
    
    \newcommand{\Rmnum}[1]
    {\expandafter\@slowromancap\romannumeral #1@}
    \makeatother

\newtheorem{thm}{Theorem}[section]

\newtheorem{prop}[thm]{Proposition}
\newtheorem{lemma}[thm]{Lemma}
\newtheorem{prob}[thm]{Problem}

\newtheorem{example}[thm]{Example}
\newtheorem{defin}[thm]{Definition}

\theoremstyle{definition}
\newtheorem{remark}[thm]{Remark}

\title[nonsymmetric $4$-class association scheme]{Every nonsymmetric $4$-class association scheme can be generated by a digraph}
\date{}

\author[Yang]{Yuefeng Yang}
\address{School of Science\\China University of Geosciences\\Beijing 100083\\China}
\email{yangyf@cugb.edu.cn}

\begin{document}

\begin{abstract}
A (di)graph $\Gamma$ generates a commutative association scheme $\mathfrak{X}$ if and only
if the adjacency matrix of $\Gamma$ generates the Bose-Mesner algebra of $\mathfrak{X}$. In \cite[Theorem 1.1]{GM}, Monzillo and Penji\'{c} proved that, except for amorphic symmetric association schemes, every
$3$-class association scheme can be generated by the adjacency matrix of a (di)graph. In this paper, we characterize when a commutative association scheme with exactly one pair of nonsymmetric relations can be generated by a digraph under certain assumptions. As an application,  we show that each nonsymmetric $4$-class association scheme can be generated by a digraph.
\end{abstract}



\maketitle
\section{Introduction}

A \emph{$d$-class association scheme} $\mathfrak{X}$ is a pair $(X,\{R_{i}\}_{i=0}^{d})$, where $X$ is a finite set, and each $R_{i}$ is a
nonempty subset of $X\times X$ satisfying the following axioms (see \cite{EB21,EB84,PHZ96,PHZ05} for a background of the theory of association schemes):
\begin{enumerate}
\item\label{as-1} $R_{0}=\{(x,x)\mid x\in X\}$ is the diagonal relation;

\item\label{as-2} $X\times X=R_{0}\cup R_{1}\cup\cdots\cup R_{d}$, $R_{i}\cap R_{j}=\emptyset~(i\neq j)$;

\item\label{as-3} for each $i$, $R_{i}^{\top}=R_{i'}$ for some $0\leq i'\leq d$, where $R_{i}^{\top}=\{(y,x)\mid(x,y)\in R_{i}\}$;

\item\label{as-4} for all $i,j,l$, the number of $z\in X$ such that $(x,z)\in R_i$ and $(z,y)\in R_j$ is constant whenever $(x,y)\in R_l$. This constant is denoted by $p_{i,j}^l$.
\end{enumerate}
A $d$-class association scheme is also called an \emph{association scheme with $d$ classes} (or even simply a \emph{scheme}). The integers $p_{i,j}^{l}$ are called the \emph{intersection numbers} of $\mathfrak{X}$. We say that $\mathfrak{X}$ is \emph{commutative} if $p_{i,j}^{l}=p_{j,i}^{l}$ for all $i,j,l$. The subsets $R_{i}$ are called the \emph{relations} of $\mathfrak{X}$. For each $i$, the integer $k_{i}:=p_{i,i'}^{0}$ is called the \emph{valency} of $R_{i}$. A relation $R_{i}$ is called \emph{symmetric} if $i=i'$, and \emph{nonsymmetric} otherwise. An association scheme is called \emph{symmetric} if all relations are symmetric, and \emph{nonsymmetric} otherwise. An association
scheme is called \emph{skew-symmetric} if the diagonal relation is the
only symmetric relation.

Let $\mathfrak{X}=(X,\{R_{i}\}_{i=0}^{d})$ be a commutative association scheme. The \emph{adjacency matrix} $A_{i}$ of $R_{i}$ is the $|X|\times |X|$ matrix whose $(x,y)$-entry is $1$ if $(x,y)\in R_{i}$, and $0$ otherwise. By the \emph{adjacency} or \emph{Bose-Mesner algebra} $\mathcal{M}$ of $\mathfrak{X}$ we mean the algebra generated by $A_{0},A_{1},\ldots,A_{d}$ over the complex field. Axioms \ref{as-1}--\ref{as-4} are equivalent to the following:
\[A_{0}=I,\quad \sum_{i=0}^{d}A_{i}=J,\quad A_{i}^{\top}=A_{i'},
\quad A_{i}A_{j}=\sum_{l=0}^{d}p_{i,j}^{l}A_{l},\]
where $I$ and $J$ are the identity and all-one matrices of order $|X|$, respectively. By \cite{GM}, $\mathcal{M}$ is a monogenic algebra, that is, there always exists a matrix $A\in {\rm Mat}_X(\mathbb{C})$ which generates $\mathcal{M}$, i.e., $\mathcal{M}=(\langle A\rangle,+,\cdot)$. We say that a matrix $A$ {\em generates} $\mathcal{M}$ if every element in $\mathcal{M}$ can be written as a polynomial in $A$.

One way to construct new association schemes is by merging  or splitting relations in an existing scheme. More precisely, a partition $\Lambda_0, \Lambda_1,\ldots, \Lambda_e$ of the index set $\{0, 1,\ldots,d\}$ is \emph{admissible} \cite{ItM91} if $\Lambda_0=\{0\}, \Lambda_i \ne \emptyset$, and $\Lambda_i' = \Lambda_j$ for some
$j\ (1\le i, j\le e)$, where $\Lambda_i' = \{\alpha'\mid\alpha\in \Lambda_i\}$. Let
 $R_{\Lambda_i} = \cup_{\alpha \in \Lambda_i} R_{\alpha}$. If $\mathfrak{Y} =(X, \{R_{\Lambda_i}\}_{ i=0}^e)$ becomes an association
scheme, it is called a \emph{fusion} scheme of ${\mathfrak X}$, while ${\mathfrak X}$ is called a \emph{fission} scheme of $\mathfrak{Y}$. In particular, $(X, \{R_0, R_i\cup R_i^{\top}\}_{i=1}^d)$ becomes an association scheme, called the {\em symmetrization} of $\mathfrak{X}$. If every admissible partition gives rise to a fusion scheme, ${\mathfrak X}$ is called {\em amorphic} (or {\em amorphous}) \cite{ItM91}.

In \cite{GM}, the authors gave the definition of a commutative association scheme generated by a digraph. A \emph{digraph} $\Gamma$ is a pair $(V(\Gamma),A(\Gamma))$ where $V(\Gamma)$ is a finite nonempty set of vertices and $A(\Gamma)$ is a set of ordered pairs ({\em arcs}) $(x,y)$ with distinct vertices $x$ and $y$. For any arc $(x,y)\in A(\Gamma)$, if $A(\Gamma)$ also contains the arc $(y,x)$, then $\{(x,y),(y,x)\}$ can be viewed as an {\em edge}. We say that $\Gamma$ is an \emph{undirected graph} or a {\em graph} if $A(\Gamma)$ is a symmetric relation. For an edge $\{(x,y),(y,x)\}$, we say that $x$ is {\em adjacent} to $y$, and also call $y$ a \emph{neighbour} of $x$. A graph is said to be {\em regular of valency} $k$ if the number of neighbour of all vertices are equal to $k$. A graph is said to be \emph{connected} if, for any vertices $x$ and $y$, there is a path from $x$ to $y$. The \emph{diameter} of a connected graph is the maximum value of the distance function in the graph.

The \emph{adjacency matrix} $A$ of $\Gamma$ is the $|V(\Gamma)|\times |V(\Gamma)|$ matrix whose $(x,y)$-entry is $1$ if $(x,y)\in A(\Gamma)$, and $0$ otherwise. The {\em eigenvalues} of a digraph are the eigenvalues of its adjacency matrix. We say that a digraph $\Gamma$ {\em generates} a commutative association scheme $\mathfrak{X}$ if and only if the adjacency matrix $A$ of $\Gamma$ generates the Bose-Mesner algebra $\mathcal{M}$ of $\mathfrak{X}$, and in symbols we write $\mathcal{M}=(\langle A\rangle,+,\cdot)$.

In \cite{GM}, the authors mentioned the following problem.

\begin{prob}\label{prob}
When can the Bose-Mesner algebra $\mathcal{M}$ of a commutative $d$-class association scheme $\mathfrak{X}$ be generated by a $01$-matrix $A$? In other words, for a given $\mathfrak{X}$, under which combinatorial and algebraic restrictions can we find a $01$-matrix $A$ such that $\mathcal{M}=(\langle A\rangle,+,\cdot)$? Moreover, since such a matrix $A$ is the adjacency matrix of some (di)graph $\Gamma$, can we describe the combinatorial structure of $\Gamma$? The vice-versa question is also of importance, i.e., what combinatorial structure does a (di)graph need to have so that its adjacency matrix will generate the Bose-Mesner algebra of a commutative $d$-class association scheme $\mathfrak{X}$?
\end{prob}

In \cite{GM}, the authors partially answer to questions posted in Problem \ref{prob}. They showed that except for amorphic symmetric association schemes, every $3$-class association scheme can be generated by a (di)graph $\Gamma$ which has $4$ distinct eigenvalues. The authors also described the combinatorial structure of a (di)graph whose adjacency matrix belongs to the Bose-Mesner algebra of a commutative association scheme. At last, the authors gave connections among weakly distance-regular digraphs in the sense of Comellas et al. (see \cite{FC04} for the details), weakly distance-regular digraphs in the sense of Wang and Suzuki (see \cite{LS,YF22,AM,HS04,KSW03,KSW04,YYF,YYF16,YYF18,YYF20,YYF22,YYF24,QZ23,QZ} for the details),
and commutative association schemes generated by a $01$-matrix $A$.

In this paper, we continue to answer the questions in Problem \ref{prob}. The first main theorem is the following result, which gives full answers to Problem \ref{prob} for certain commutative association schemes with exactly one pair of nonsymmetric relations.

\begin{thm}\label{one pair}
Let $\mathfrak{X}=(X,\{R_i\}_{i=0}^{d+1})$ be a commutative association scheme with $R_d^{\top}=R_{d+1}$ and $R_i^{\top}=R_i$ for $0\leq i\leq d-1$, and $\tilde{\mathfrak{X}}$ be its symmetrization. Suppose that the graph $(X,R_d\cup R_{d+1})$ generates $\tilde{\mathfrak{X}}$. Then the following hold.
\begin{enumerate}
\item\label{one pair-1} The adjacency matrix of $(X,R_d)$ has exactly $d+2$ distinct eigenvalues.

\item\label{one pair-2} The digraph $(X,R_d)$ generates the association scheme $\mathfrak{X}$.
\end{enumerate}
\end{thm}

To state our second main theorem, we prepare some more basic notations.  Let $\mathfrak{X}=(X,\{R_{i}\}_{i=0}^{d})$ be a commutative association scheme. Since the Bose-Mesner algebra $\mathcal{M}$ of $\mathfrak{X}$ consists of commuting normal matrices, it has a second basis consisting of primitive idempotents $E_0= J/|X|, \dots, E_d$. The integers $m_{i}=\textrm{rank}E_{i}$ are called the \emph{multiplicities} of $\mathfrak{X}$, and $m_{0}=1$ is said to be the {\em trivial multiplicity}. For $0\leq i,j\leq d$, there exists a complex scalar $p_i(j)$ such that $A_iE_j=p_i(j)E_j$ (moreover, $p_i(j)$ is the eigenvalue of $A_i$ on the eigenspace $V_j$), where $E_j$ is the orthogonal projector of $\mathbb{C}^{|X|}$ onto the space $V_j:= E_j\mathbb{C}^{|X|}$. The change-of-basis matrix $P$ is defined by
\begin{align}
  A_i = \sum_{j=0}^d (P)_{ji} E_j.\nonumber
\end{align}
The $(d+1)\times (d+1)$ matrix $P$ is called the {\em character table} or {\em first eigenmatrix} of ${\mathfrak X}$. According to \cite{EB84}, the character table $P$ has the following form
\begin{align}
\left(\begin{array}{cccc}
1 & k_1 & \cdots & k_d \\
1 & p_1(1) & \cdots & p_d(1)\\
\vdots & \vdots & \ddots & \vdots\\
1 & p_1(d) & \cdots & p_d(d)
\end{array}\right),
\label{character table}
\end{align}
where the scalars $k_i, p_i(1),\ldots,p_i(d)$ are the eigenvalues (not necessarily pairwise distinct) of $A_i$ on $V_0, V_1,\ldots,V_d$, respectively.

The second main theorem concerns a symmetric amorphic association scheme $\tilde{\mathfrak{X}}=(X,\{\tilde{R}_i\}_{i=0}^d)$. Clearly, for each proper subset $\Lambda$ of $\{1,2,\ldots,d\}$, the graph $(X,\tilde{R}_{\Lambda})$ is strongly regular. According to \cite[Theorem 1 and Proposition 2]{ERD10}, via permutations of the rows and columns, we may assume that the character table $\tilde{\mathfrak{X}}$ has the following form:
\begin{align}\label{amorphic}
\tilde{P}=\begin{pmatrix}
     1 & k_1 & k_2 & k_3 & \cdots & k_d \\
     1 & b_1 & a_2 & a_3 & \cdots & a_d\\
     1 & a_1 & b_2 & a_3 & \cdots & a_d\\
     1 & a_1 & a_2 & b_3 & \cdots & a_d\\
     \vdots &\vdots & \vdots & \vdots & \ddots  & \vdots\\
     1 & a_1 & a_2 & a_3 & \cdots & b_d
\end{pmatrix},
\end{align}
where $a_i$ and $b_i$ are distinct real numbers for each $i\in\{1,2,\ldots,d\}$.

The second main theorem is as follows, which fully answers Problem \ref{prob} for commutative association schemes with exactly one pair of nonsymmetric relations whose symmetrizations are amorphic.

\begin{thm}\label{main-general}
Let $\mathfrak{X}=(X,\{R_i\}_{i=0}^{d+1})$ be a commutative association scheme with $R_{d}^{\top}=R_{d+1}$ and $R_{i}^{\top}=R_i$ for $0\leq i\leq d-1$, and $\tilde{\mathfrak{X}}$ be its symmetrization with primitive idempotents $\tilde{E}_0,\tilde{E}_1,\ldots,\tilde{E}_d$. Suppose that $\tilde{\mathfrak{X}}$ is amorphic. Assume that the character table $\tilde{P}$ of $\tilde{\mathfrak{X}}$ has form as \eqref{amorphic}. Then the following hold.
\begin{enumerate}
\item\label{main-general-1} $\mathfrak{X}$ can be generated by a digraph if and only if $d\leq 3$, or $d=4$ and $\tilde{E}_d$ is also a primitive idempotent of $\mathfrak{X}$.

\item\label{main-general-2} If $\mathfrak{X}$ can be generated by a digraph, then there exists $i\in\{1,2,\ldots,d-1,d\}$ such that $(X,R_i\cup R_d)$ has $d+2$ distinct eigenvalues and generates $\mathfrak{X}$.
\end{enumerate}
\end{thm}

As an application of Theorems \ref{one pair} and \ref{main-general}, we get full answers to questions posted
in Problem \ref{prob} for nonsymmetric $4$-class association schemes, and obtain the third main result as follows.

\begin{thm}\label{main}
Let $\mathfrak{X}=(X,\{R_0,R_1,R_2,R_3,R_4\})$ be a nonsymmetric association scheme with $R_3^{\top}=R_4$. Then there exists an integer $i\in\{1,2,3\}$ such that the following hold.
\begin{enumerate}
\item\label{main1-1} The adjacency matrix of $(X,R_i\cup R_3)$ has exactly $5$ distinct eigenvalues.

\item\label{main1-2} The digraph $(X,R_i\cup R_3)$ generates the association scheme $\mathfrak{X}$.
\end{enumerate}
\end{thm}

The remainder of this paper is organized as follows. In Section 2, we recall some basic definitions and results of algebraic graph theory and association schemes. In Section 3, we give proofs of Theorems \ref{one pair} and \ref{main-general}. In Section 4, we prove Theorem \ref{main} based on Theorems \ref{one pair} and \ref{main-general}.

\section{Preliminaries}

In this section, we shall recall some definitions and results of algebraic graph theory and association schemes which are used frequently in this paper.

\subsection{Basic results}

First, we recall some results of algebraic graph theory.

\begin{lemma}\label{diameter}
{\rm (\cite[Lemma 8.12.1]{CG01})} If $\Gamma$ is a connected graph with diameter $d$, then $\Gamma$ has at least $d+1$ distinct eigenvalues.
\end{lemma}

\begin{lemma}\label{regular}
{\rm (\cite[Lemma 3.2.1]{AEB98})} Let $\Gamma$ be a regular graph of valency $k$. If $\theta$ is an eigenvalue of $\Gamma$, then $k\geq\theta$.
\end{lemma}

\begin{prop}\label{multiplicity}
{\rm (\cite[Proposition 3.2.2]{AEB98})} Let $\Gamma$ be a regular graph of valency $k$. The multiplicity of the eigenvalue $k$ is the number of connected components of $\Gamma$.
\end{prop}

In the remainder of this subsection, we give some results concerning association scheme, and we always assume that $\mathfrak{X}=(X,\{R_{i}\}_{i=0}^{d})$ is a commutative association scheme with primitive idempotents $E_0,E_1\ldots,E_d$ and character table $P$.

\begin{lemma}\label{row}
{\rm (\cite[Corollary 2.6]{GM})} Let $V_i=E_i\mathbb{C}^{|X|}$ for $0\leq i\leq d$. Then the sum of the entries of the $V_i$ row in $P$ is equal to $0$ for $1\leq i\leq d$.
\end{lemma}

In the rest of this paper, the following theorem, known as the Bannai-Muzychuk criterion for fusion
schemes, will be used repeatedly \cite{EB93,ItM91}.

\begin{thm}\label{Bannai-Muzychuk}
Let $\{\Lambda_i\}_{i=0}^e$ be an admissible partition of the index set $\{0, 1,\ldots,d\}$. Then $(X,\{R_{\Lambda_i}\}_{i=0}^e)$ is a fusion scheme if and only if there exists a partition $\{\Lambda_i^*\}_{i=0}^e$ of $\{0, 1,\ldots,d\}$ with $\Lambda_0^*=\{0\}$ such that each $(\Lambda_i^*,\Lambda_j)$ block of the character table $P$ has constant row sum. Moreover, this constant is the $(i,j)$-entry of the character table of this fusion scheme.
\end{thm}

%For $i = 0,1,\dots,d$, the {\em $i$-th intersection matrix} $B_i$ is defined to be the $(d+1) \times (d+1)$ matrix whose $(j,l)$ entry is $p_{i,j}^l$. The character table $P$ determines matrices $B_i$, and vice versa.

%\begin{thm}\label{commutative}
%{\rm (\cite{Hig75})}~Association schemes with at most four classes are commutative.
%\end{thm}

We close this subsection with the following result which gives a necessary and sufficient condition for a $01$-matrix generating $\mathfrak{X}$.

\begin{thm}\label{generates}
{\rm (\cite[Corollary 2.15]{GM})}~Let $\mathcal{M}$ be the Bose-Mesner algebra of $\mathfrak{X}$ and $A$ be a $01$-matrix in $\mathcal{M}$. Then $A$ generates $\mathcal{M}$ if and only if $A$ has $d+1$ distinct eigenvalues.
\end{thm}

\subsection{Symmetric 2-class association schemes: strongly regular graphs}

The concept of a strongly regular graph is essentially the same as that of a symmetric $2$-class association scheme. We refer the reader to \cite{AEB98,AEB22,PJC99,CG01,JJS79} for further details on the general theory of strongly regular graphs.

Let $(X,R)$ be a regular graph of valency $k$ with $n$ vertices that is neither complete nor empty. Then $(X,R)$ is called {\em strongly regular} if any pair of adjacent vertices have $\lambda$ common neighbors, and any two distinct nonadjacent vertices have $\mu$ common neighbors. We say this is a strongly regular graph with parameters $(n,k,\lambda,\mu)$. It is easy to verify that $\mathfrak{X} = (X,\{R_0, R, \bar{R}\})$ is a symmetric association scheme,
where $R_0=\{(x,x)\mid x\in X\}$ and $\bar{R}=\{(x,y)\in X\times X\mid (x,y)\notin R\}$.

On the other hand, any symmetric $2$-class association scheme $\mathfrak{X}=(X,\{R_0,R_1,R_2\})$ is described as a pair of strongly regular graphs $(X,R_1)$ and $(X,R_2)$ with parameters $(n,k_1,p_{1,1}^1,p_{1,1}^2)$ and  $(n,k_2,p_{2,2}^2,p_{2,2}^1)$, respectively. Note that $n=k_1+k_2+1$.

Now we consider the strongly regular graphs which are not connected. %Before that, we need one more concept. A {\em clique} $C$ of a graph $\Gamma$ is an induced subgraph of $\Gamma$ such that every two distinct vertices of $C$ are adjacent (i.e., aclique of $\Gamma$ is a complete subgraph of $\Gamma$). The number of vertices of $C$ is called the size of the clique $C$.

\begin{thm}\label{disconnected}
{\rm (\cite[Theorem 3.11]{PJC99})}~A disconnected strongly
regular graph is a disjoint union of at least two complete graphs of the
same size. Conversely, if a graph is a disjoint union of at least two
complete graphs of the same size, then it is a disconnected strongly
regular graph.
\end{thm}

\begin{prop}\label{disconnected eigen}
Let $(X,R)$ be a strongly regular graph of valency $k$. Then the following are equivalent:
\begin{enumerate}
\item\label{disconnected eigen-1} $(X,R)$ is not connected;

\item\label{disconnected eigen-2} $(X,R)$ has $-1$ as its eigenvalue;

\item\label{disconnected eigen-3} $(X,R)$ has exactly two eigenvalues: $k$ and $-1$.
\end{enumerate}
\end{prop}
\begin{proof}
It is immediate from Theorem \ref{disconnected} and \cite[Theorem 1.3.1]{AEB98}.
\end{proof}

Now let $\Gamma$ be a connected strongly regular graph with parameters $(n,k,\lambda,\mu)$. Counting in two different ways the edges between vertices which are adjacent and nonadjacent
to a fixed $x\in X$, we get the well-known identity
\begin{align}\label{k2}
k(k-\lambda-1)=\mu(n-k-1).
\end{align}

The following result concerns the eigenvalues of a connected strongly regular graph.

\begin{thm}\label{three eigenvalue}
{\rm (\cite[Theorem 1.3.1]{AEB98} and \cite[Lemma 10.2.1]{CG01})}~A connected regular graph is strongly regular if and only if it has exactly three distinct eigenvalues.
\end{thm}

By Proposition \ref{multiplicity}, Theorem \ref{three eigenvalue}, and \cite[Section 10.2]{CG01}, $\Gamma$ has exactly three eigenvalues: $k$, $r$, and $s$ of multiplicity, respectively, 1, $m_1$, and $m_2$, where
\begin{align}
r,s&=\frac{(\lambda-\mu)\pm\sqrt{(\lambda-\mu)^2+4(k-\mu)}}{2},\label{eigenvalue}\\
m_{1},m_2&=\frac{1}{2}\Big[(n-1)\mp\frac{2k-(n-1)(\mu-\lambda)}{\sqrt{(\lambda-\mu)^2+4(k-\mu)}})\Big].\label{mulitiplicity}
\end{align}
By \eqref{eigenvalue}, Lemma \ref{regular}, and Proposition \ref{disconnected eigen}, we have $k>r\geq0>s\neq-1$.
We say that $\Gamma$ is a {\em conference graph} if $m_1=m_2$.

\begin{prop}\label{integer}
{\rm (\cite[Lemma 10.3.3]{CG01})}~Let $\Gamma$ be a strongly regular graph with eigenvalues $k$, $r$, and $s$. Then either
\begin{enumerate}
\item\label{integer-1} $\Gamma$ is a conference graph, or

\item\label{integer-2} $r$ and $s$ are integers.
\end{enumerate}
\end{prop}

On the other hand, we can express the parameters of a connected strongly regular graph $\Gamma$ starting from its eigenvalues and multiplicities: $n=1+m_1+m_2$,
\begin{align}
\lambda&=k+rs+r+s,\label{lambda}\\
\mu&=k+rs.\label{mu}
\end{align}

\section{Proofs of Theorems \ref{one pair} and \ref{main-general}}

In this section, we always assume that $\mathfrak{X}=(X,\{R_i\}_{i=0}^{d+1})$ is a commutative association scheme with primitive idempotents $E_0,E_1,\ldots,E_{d+1}$, where $R_d^{\top}=R_{d+1}$ and $R_i^{\top}=R_i$ for $0\leq i\leq d-1$. Let $\tilde{\mathfrak{X}}=(X,\{\tilde{R}_i\}_{i=0}^d)$ be the symmetrization of $\mathfrak{X}$ with primitive idempotents $\tilde{E}_0,\tilde{E}_1,\ldots,\tilde{E}_{d}$, where $\tilde{R}_d=R_d\cup R_{d+1}$ and $\tilde{R}_i=R_i$ for $0\leq i\leq d-1$.

The following result determines the character table of a commutative association scheme with exactly one pair of nonsymmetric relations.

\begin{thm}\label{fission scheme}
{\rm (\cite[Theorem 2.1]{GLC06})}~Suppose that the character table $\tilde{P}$ of $\tilde{\mathfrak{X}}$ has form as \eqref{character table}. If $\tilde{E}_1=E_1+E_2$ and $\tilde{E}_i=E_{i+1}$ for $2\leq i\leq d$, then
the character table of $\mathfrak{X}$ has the following form:
\begin{align}
P=\left(\begin{array}{ccccc|cc}
1 & k_1     & k_2   & \cdots & k_{d-1} & k_d/2     & k_d/2\\\hline
1 & p_1(1)  & p_2(1)& \cdots & p_{d-1}(1) & \rho      & \bar{\rho}\\
1 & p_1(1)  & p_2(1)& \cdots & p_{d-1}(1) & \bar{\rho}& \rho\\\hline
1 & p_1(2)  & p_2(2)& \cdots & p_{d-1}(2) & p_d(2)/2  & p_d(2)/2\\
1 & p_2(3)  & p_3(3)& \cdots & p_{d-1}(3)     & p_d(3)/2  & p_d(3)/2\\
\vdots& \vdots& \vdots     &\vdots   & \vdots& \ddots & \vdots\\
1 & p_2(d)  & p_3(d)& \cdots & p_{d-1}(d) & p_d(d)/2  & p_d(d)/2
\end{array}\right),
\nonumber
\end{align}
where $\rho=(p_d(1)+\sqrt{a})/2$ and $a<0$.
\end{thm}

Now we are ready to give a proof of Theorem \ref{one pair}.

\begin{proof}[Proof of Theorem \ref{one pair}]
Suppose that the character table $\tilde{P}$ of $\tilde{\mathfrak{X}}$ has form as \eqref{character table}. By Theorem \ref{Bannai-Muzychuk}, we may assume $\tilde{E}_1=E_1+E_2$ and $\tilde{E}_j=E_{j+1}$ for $2\leq j\leq d$. Since $(X,\tilde{R}_d)$ generates $\tilde{\mathfrak{X}}$, from Theorem \ref{generates}, $(X,\tilde{R}_d)$ has $d+1$ distinct eigenvalues. In view of \eqref{character table}, all distinct eigenvalues of $(X,\tilde{R}_d)$ are $k_d,p_d(1),p_d(2),\ldots,p_d(d)$. By Theorem \ref{fission scheme}, all distinct eigenvalues of $(X,R_d)$ are $k_d/2,p_d(2)/2,p_d(3)/2,\ldots,p_d(d)/2$, and $(p_d(1)\pm\sqrt{a})/2$ with $a<0$. Then \ref{one pair-1} holds, and \ref{one pair-2} is valid from Theorem \ref{generates}.
\end{proof}

In the remainder of this section, we always assume that $\tilde{\mathfrak{X}}$ is an amorphic association scheme with the character table $\tilde{P}$ having form as \eqref{amorphic}. To give a proof of Theorem \ref{main-general}, we need several auxiliary results.

\begin{lemma}\label{a b}
Let $d\geq3$. If $1\leq i<j\leq d$, then $a_i+b_j=a_j+b_i$.
\end{lemma}
\begin{proof}
By \eqref{amorphic}, all eigenvalues of the graph $(X,\tilde{R}_i\cup\tilde{R}_j)$ are $k_{i}+k_j$, $a_i+a_j$, $a_i+b_j$, and $b_i+a_j$. Since $a_i\neq b_i$ and $a_j\neq b_j$, we have $a_i+a_j\notin\{a_i+b_j,b_i+a_j\}$.

Suppose to the contrary that $a_i+b_j\neq b_i+a_j$. Then the graph $(X,\tilde{R}_i\cup\tilde{R}_j)$ has at least three distinct eigenvalues. Since $(X,\tilde{R}_i\cup\tilde{R}_j)$ is a strongly regular graph of valency $k_i+k_j$, from Proposition \ref{disconnected eigen}, $(X,\tilde{R}_i\cup\tilde{R}_j)$ is connected, which implies that $k_i+k_j$ is an eigenvalue of $(X,\tilde{R}_i\cup\tilde{R}_j)$ with multiplicity $1$ by Proposition \ref{multiplicity}. It follows that $k_i+k_j\notin\{a_i+a_j,a_i+b_j,b_i+a_j\}$. Then the connected strongly regular graph $(X,\tilde{R}_i\cup\tilde{R}_j)$ has four distinct eigenvalues, contrary to Theorem \ref{three eigenvalue}.
\end{proof}

\begin{lemma}\label{d d+1}
Let $\Lambda$ be a nonempty subset of $\{1,2,\ldots,d+1\}$ with $d\geq2$. If the digraph $(X,R_{\Lambda})$ has $d+2$ distinct eigenvalues, then exactly one element in $\{d,d+1\}$ belongs to $\Lambda$.
\end{lemma}
\begin{proof}
Since $(X,\tilde{R}_{\Omega})$ is strongly regular for all proper subset $\Omega$ of $\{1,2,\ldots,d\}$, from Proposition \ref{disconnected eigen} and Theorem \ref{three eigenvalue}, $(X,\tilde{R}_{\Omega})$ has at most three distinct eigenvalues. Thus, the desired result follows.
\end{proof}

\begin{lemma}\label{i j}
Let $V_i=E_i\mathbb{C}^{|X|}$ for $0\leq i\leq d+1$. Suppose $E_i=\tilde{E}_r$ and $E_j=\tilde{E}_s$ with $r,s\in\{1,2,\ldots,d-1\}$ for $1\leq i<j\leq d+1$. %Suppose that the eigenvalues of $A_r$ (resp. $A_{s}$) on the eigenspaces $V_i$ and $V_j$ are $a_r$ and $b_r$ (resp. $b_{s}$ and $a_s$) respectively with $1\leq r,s\leq d-1$.
Then $A_r$ and $A_s$ are the only two adjacency matrices of $\mathfrak{X}$ having distinct eigenvalues on the eigenspaces $V_i$ and $V_j$.
\end{lemma}
\begin{proof}
Since $E_i=\tilde{E}_r$ and $E_j=\tilde{E}_s$, from \eqref{amorphic} and Theorems \ref{Bannai-Muzychuk}, \ref{fission scheme}, the eigenvalues of $A_r$ (resp. $A_{s}$) on the eigenspaces $V_i$ and $V_j$ are $b_r$ and $a_r$ (resp. $a_{s}$ and $b_s$) respectively. By \eqref{amorphic} and Theorems \ref{Bannai-Muzychuk}, \ref{fission scheme} again, the eigenvalues of $A_{h}$ on the eigenspaces $V_i$ and $V_j$ are both $a_h$ for $h\in\{1,2,\ldots,d-1\}\setminus\{r,s\}$. Since $d\notin\{r,s\}$, the eigenvalues of $A_{l}$ on the eigenspaces $V_i$ and $V_j$ are both $a_d/2$ for $l\in\{d,d+1\}$. This completes the proof.
\end{proof}

\begin{lemma}\label{contradiction}
Let $V_i=E_i\mathbb{C}^{|X|}$ for $0\leq i\leq d+1$. Suppose that $\mathfrak{X}$ can be generated by a digraph. Then there exist at most two distinct integers $i,j\in\{1,2,\ldots,d+1\}$ such that the eigenvalues of $A_d$ on the eigenspaces $V_i$ and $V_j$ are the same.
\end{lemma}
\begin{proof}
Let $\Gamma$ be a digraph generating $\mathfrak{X}$. Since the adjacency matrix of $\Gamma$ belongs to the Bose-Mesner algebra of $\mathfrak{X}$, there exists a nonempty subset $\Lambda$ of $\{1,2,\ldots,d+1\}$ such that $\Gamma=(X,R_{\Lambda})$. By Theorem \ref{generates}, $(X,R_{\Lambda})$ has $d+2$ distinct eigenvalues.

Assume the contrary, namely, there exist three distinct positive integers $i,j,l\in\{1,2,\ldots,d+1\}$ such that all the eigenvalues of $A_d$ on the eigenspaces $V_i,V_j$, and $V_l$ are the same. By \eqref{amorphic} and Theorems \ref{Bannai-Muzychuk}, \ref{fission scheme}, we have $E_i=\tilde{E}_t$, $E_j=\tilde{E}_r$, and $E_l=\tilde{E}_s$ for some distinct integers $t,r,s\in\{1,2,\ldots,d-1\}$, and the eigenvalues of $A_d$ on the eigenspaces $V_i,V_j$, and $V_l$ are all $a_d/2$. Moreover, the eigenvalues of $A_{d+1}$ on the eigenspaces $V_i,V_j$, and $V_l$ are also all $a_d/2$. By Lemma \ref{d d+1}, we may assume $d\in\Lambda$. Then $d+1\notin\Lambda$.

Since the digraph $(X,R_{\Lambda})$ has $d+2$ distinct eigenvalues, the eigenvalues of $\sum_{h\in\Lambda}A_{h}$ on the eigenspaces $V_i$ and $V_j$ are distinct. Since $E_i=\tilde{E}_t$ and $E_j=\tilde{E}_r$, from Lemma \ref{i j}, we may assume that $r\in\Lambda$. In view of \eqref{amorphic}, the eigenvalues of $A_r$ on the eigenspaces $V_i$ and $V_j$ are $a_r$ and $b_r$, respectively. Since $E_l=\tilde{E}_s$, from \eqref{amorphic}, the eigenvalue of $A_r$ on the eigenspace $V_l$ is $a_r$. Thus, the eigenvalues of $A_r+A_d$ on the eigenspaces $V_i,V_j$, and $V_l$ are $a_r+a_d/2$, $b_r+a_d/2$, and $a_r+a_d/2$, respectively.

Since the digraph $(X,R_{\Lambda})$ has $d+2$ distinct eigenvalues, the eigenvalues of $\sum_{h\in\Lambda}A_{h}$ on the eigenspaces $V_i$ and $V_l$ are distinct. Since $E_i=\tilde{E}_t$ and $E_l=\tilde{E}_s$, from Lemma \ref{i j}, we may assume that $s\in\Lambda$. In view of \eqref{amorphic}, the eigenvalues of $A_s$ on the eigenspaces $V_i$ and $V_l$ are $a_s$ and $b_s$, respectively. Since $E_j=\tilde{E}_r$, from \eqref{amorphic}, the eigenvalue of $A_s$ on the eigenspace $V_j$ is $a_s$. Thus, the eigenvalues of $A_r+A_s+A_d$ on the eigenspaces $V_i,V_j$, and $V_l$ are $a_r+a_s+a_d/2$, $b_r+a_s+a_d/2$, and $a_r+b_s+a_d/2$, respectively.

By Lemma \ref{a b}, the eigenvalues of $A_r+A_s+A_d$ on the eigenspaces $V_j$ and $V_l$ are the same. Since $E_j=\tilde{E}_r$ and $E_l=\tilde{E}_s$, from Lemma \ref{i j}, the eigenvalues of $A_h$ on the eigenspaces $V_j$ and $V_l$ are the same for all $h\in\Lambda\setminus\{s,r\}$. It follows that the eigenvalues of $\sum_{h\in\Lambda}A_h$ on the eigenspaces $V_j$ and $V_l$ are also the same, contrary to the fact that $(X,R_{\Lambda})$ has $d+2$ distinct eigenvalues.
\end{proof}

\begin{lemma}\label{two disconnected}
If $d\geq3$, then $a_i,b_i<k_i$ for $1\leq i\leq d$.
\end{lemma}
\begin{proof}
By \cite[Theorem 1]{ERD10}, $(X,\tilde{R}_i)$ is a connected strongly regular graph for $1\leq i\leq d$. By \eqref{amorphic}, Lemma \ref{regular}, and Theorem \ref{three eigenvalue}, we have $a_i,b_i<k_i$.
\end{proof}

\begin{prop}\label{E3 notin}
Let $d=3$. If $\tilde{E}_3\notin\{E_1,E_2,E_3,E_4\}$, then $(X,R_i\cup R_3)$ has $5$ distinct eigenvalues for $i\in\{1,2\}$.
\end{prop}
\begin{proof}
By Theorem \ref{Bannai-Muzychuk}, we may assume $\tilde{E}_3=E_3+E_4$ and $\tilde{E}_h=E_h$ for $0\leq h\leq 2$. In view of \eqref{amorphic} and Theorem \ref{fission scheme}, the character table of $\mathfrak{X}$ has the form
\begin{align}
P =\left(\begin{array}{ccccc}
1 & k_1 & k_2 & k_3/2 & k_3/2 \\
1 & b_1 & a_2 & a_3/2 & a_3/2\\
1 & a_1 & b_2 & a_3/2 & a_3/2\\
1 & a_1 & a_2 & (b_3+\sqrt{a})/2 & (b_3-\sqrt{a})/2\\
1 & a_1 & a_2 & (b_3-\sqrt{a})/2 & (b_3+\sqrt{a})/2
\end{array}\right)\nonumber
\end{align}
with $a<0$. Then eigenvalues of $(X,R_i\cup R_3)$ are $k_i+k_3/2$, $a_i+a_3/2$, $b_i+a_3/2$, and $(2a_i+b_3\pm\sqrt{a})/2$. Since $a_i\neq b_i$, from Lemma \ref{two disconnected}, $(X,R_i\cup R_3)$ has $5$ distinct eigenvalues.
%Since both $(X,\tilde{R}_i)$ and $(X,\tilde{R}_3)$ are strongly regular, we have $k_i\geq a_i,b_i$ and $k_3\geq a_3$. Suppose $k_i\in\{a_i,b_i\}$ and $k_3=a_3$. By Lemma \ref{two disconnected}, one gets $k_i=b_i$. Since a connected strongly regular graph has exactly three distinct eigenvalues, both $(X,\tilde{R}_i)$ and $(X,\tilde{R}_3)$ are not connected. By Proposition \ref{disconnected eigen}, we have $a_i=b_3=-1$, which implies $a_i+b_3=-2$ and $b_i+a_3=k_i+k_3$, contrary to Lemma \ref{a b}. Then $k_i\notin\{a_i,b_i\}$ or $k_3>a_3$. Since $a_i\neq b_i$, $k_i+k_3/2$, $a_i+a_3/2$ and $b_i+a_3/2$ are distinct.
%
%This completes the proof of this lemma.
\end{proof}

\begin{prop}\label{E3 in}
Suppose $\tilde{E}_d\in\{E_1,E_2,\ldots,E_{d+1}\}$. The following hold.
\begin{enumerate}
\item\label{E3 in-1} If $d=3$, then $(X,R_i\cup R_3)$ has $5$ distinct eigenvalues for all $i\in\{1,2,3\}$.

\item\label{E3 in-2} If $d=4$, then $(X,R_i\cup R_4)$ has $6$ distinct eigenvalues for all $i\in\{2,3\}$.
\end{enumerate}
\end{prop}
\begin{proof}
In view of \eqref{amorphic} and Theorem \ref{Bannai-Muzychuk}, we may assume $\tilde{E}_1=E_1+E_2$ and $\tilde{E}_h=E_{h+1}$ for $2\leq h\leq d$. By Theorem \ref{fission scheme}, the character table of $\mathfrak{X}$ has the form
\begin{align}\label{d=4}
P =\left(\begin{array}{ccccccc}
1 & k_1 & k_2 & \cdots & k_{d-1} & k_d/2 & k_d/2 \\
1 & b_1 & a_2 & \cdots & a_{d-1} & (a_d+\sqrt{a})/2 & (a_d-\sqrt{a})/2\\
1 & b_1 & a_2 & \cdots & a_{d-1} &(a_d-\sqrt{a})/2 & (a_d+\sqrt{a})/2\\
1 & a_1 & b_2 & \cdots & a_{d-1} & a_d/2 & a_d/2\\
\vdots &\vdots & \vdots & \ddots  & \vdots & \vdots &  \vdots\\
1 & a_1 & a_2 & \cdots & b_{d-1} & a_d/2 & a_d/2\\
1 & a_1 & a_2 & \cdots & a_{d-1} & b_d/2 & b_d/2
\end{array}\right)
\end{align}
with $a<0$. To prove this proposition, we need to consider the cases $d=3$ and $d=4$.

\textbf{Case 1.} $d=3$.

Since $a_3\neq b_3$, from \eqref{d=4} and Lemma \ref{two disconnected},  all distinct eigenvalues of $(X,R_3)$ are $k_3/2$, $(a_3\pm\sqrt{a})/2$, $a_3/2$, and $b_3/2$.

By \eqref{d=4}, $(X,R_1\cup R_3)$ has eigenvalues $k_1+k_3/2$, $(2b_1+a_3\pm\sqrt{a})/2$, $a_1+a_3/2$, and $a_1+b_3/2$. Since $a_3\neq b_3$, from Lemma \ref{two disconnected}, $(X,R_1\cup R_3)$ has $5$ distinct eigenvalues.

By \eqref{d=4}, $(X,R_2\cup R_3)$ has eigenvalues $k_2+k_3/2$, $(2a_2+a_3\pm\sqrt{a})/2$, $b_2+a_3/2$, and $a_2+b_3/2$. Since $a_3\neq b_3$ and $b_2+a_3=a_2+b_3$ from Lemma \ref{a b}, we have $b_2+a_3/2\neq a_2+b_3/2$. By Lemma \ref{two disconnected}, $(X,R_2\cup R_3)$ has $5$ distinct eigenvalues. Thus, (i) is valid.

\textbf{Case 2.} $d=4$.

By \eqref{d=4}, $(X,R_i\cup R_4)$ has eigenvalues $k_i+k_4/2$, $a_i+a_4/2$, $b_i+a_4/2$, $a_i+b_4/2$, and $(2a_i+a_4\pm\sqrt{a})/2$. Since $a_h\neq b_h$ for $1\leq h\leq d$ and $b_i+a_4=a_i+b_4$ from Lemma \ref{a b}, $a_i+a_4/2$, $b_i+a_4/2$, and $a_i+b_4/2$ are distinct. By Lemma \ref{two disconnected}, $(X,R_i\cup R_4)$ has $6$ distinct eigenvalues. Thus, (ii) is also valid.
%In view of Lemma \ref{two disconnected}, we get $k_i+k_4/2\in\{b_i+a_4/2,a_i+b_4/2\}$ for $i\in\{2,3\}$. Since $(X,\tilde{R}_h)$ is a strongly regular graph with eigenvalues $k_h,a_h,b_h$ for $1\leq h\leq d$, one has $k_h\geq a_h,b_h$, which implies $(k_i,k_4)=(b_i,a_4)$ or $(a_i,b_4)$. Since a connected strongly regular graph has exactly three distinct eigenvalues, both $(X,\tilde{R}_i)$ and $(X,\tilde{R}_4)$ are not connected. By Proposition \ref{disconnected eigen}, we have $a_i=b_4=-1$ or $b_i=a_4=-1$, which implies $\{b_i+a_4,a_i+b_4\}=\{k_i+k_4,-2\}$, contrary to Lemma \ref{a b}.
\end{proof}

Now we are ready to prove Theorem \ref{main-general}.

\begin{proof}[Proof of Theorem \ref{main-general}]
If $d=1$, from Theorems \ref{Bannai-Muzychuk} and \ref{fission scheme}, then the graph $(X,R_j)$ has exactly $3$ distinct eigenvalues for $j\in\{1,2\}$, which implies that $(X,R_j)$ generates $\mathfrak{X}$ by Theorem \ref{generates}. Then the sufficiency of \ref{main-general-1}, and \ref{main-general-2} are both valid from \cite[Theorem 3.8]{GM}, Propositions \ref{E3 notin}, \ref{E3 in}, and Theorem \ref{generates}.

It suffices to prove the necessity of \ref{main-general-1}. Let $\Gamma$ be a digraph generating $\mathfrak{X}$. Since the adjacency matrix of $\Gamma$ belongs to the Bose-Mesner algebra of $\mathfrak{X}$, there exists a proper subset $\Lambda$ of $\{1,2,\ldots,d+1\}$ such that $\Gamma=(X,R_{\Lambda})$. Assume the contrary, namely, $d\geq5$, or $d=4$ and $\tilde{E}_4\notin\{E_1,E_2,E_3,E_4\}$. Let $V_h=E_h\mathbb{C}^{|X|}$ for $0\leq h\leq d+1$.

Suppose $\tilde{E}_d\notin\{E_1,E_2,\ldots,E_{d+1}\}$. By Theorem \ref{Bannai-Muzychuk}, without loss of generality, we may assume $\tilde{E}_{d}=E_d+E_{d+1}$ and $\tilde{E}_h=E_h$ for $0\leq i\leq d-1$. Since $d\geq4$, from \eqref{amorphic} and Theorem \ref{fission scheme}, the eigenvalues of $A_{d}$ on the eigenspaces $V_1$, $V_2$, and $V_3$ are all $a_d/2$, contrary to Lemma \ref{contradiction}. Then, $d\geq5$ and $\tilde{E}_d\in\{E_1,E_2,\ldots,E_{d+1}\}$.

By Theorem \ref{Bannai-Muzychuk}, we may assume $\tilde{E}_{1}=E_1+E_2$ and $\tilde{E}_h=E_{h+1}$ for $2\leq i\leq d$. Since $d\geq5$, from \eqref{amorphic} and Theorem \ref{fission scheme}, the eigenvalues of $A_{d}$ on the eigenspaces $V_3$, $V_4$, and $V_5$ are all $a_d/2$, contrary to Lemma \ref{contradiction}. Thus, the necessity of \ref{main-general-1} is also valid.
\end{proof}

\section{Proof of Theorem \ref{main}}

In this section, we always assume that $\mathfrak{X}=(X,\{R_0,R_1,R_2,R_3,R_4\})$ is an association scheme with $R_3^{\top}=R_4$. By a theorem of Higman \cite{Hig75}, association schemes with at most four classes are commutative. Next, we divide this section into two subsections according to whether $\mathfrak{X}$ is skew-symmetric or not.

\subsection{$\mathfrak{X}$ is skew-symmetric}

Since $\mathfrak{X}$ is skew-symmetric, we have $R_1^{\top}=R_2$. At first, we recall the character table of $\mathfrak{X}$.

\begin{thm} \label{t:rs20}
{\rm(\cite[Theorem 5]{Ma10})}~Let $\tilde{\mathfrak{X}} =(X,\{R_0, R_1\cup R_2 , R_3\cup R_4\})$
be the symmetrization of $\mathfrak{X}$ and
\begin{align}
\tilde{P} = \left(\begin{array}{ccc}
1 & k_1 & k_2 \\
1  & r_1 & t_1 \\
1 & r_2 & t_2
\end{array}\right)
\begin{array}{c} 1 \\ m_1 \\ m_2 \end{array}\label{sym character table}
\end{align}
the character table of  $\tilde{\mathfrak{X}}$, where $1, m_1$, and $m_2$ are the multiplicities of the corresponding row eigenvalues. Then the character table of $\mathfrak{X}$ has the following form$:$
$$
P= \left(\begin{array}{ccccc}
      1 & k_1/2  & k_1/2 & k_2/2 & k_2/2\\
      1  & \rho  & \bar{\rho} & \tau & \bar{\tau} \\
	 1 & \sigma  & \bar{\sigma} & \omega & \bar{\omega} \\
	 1 & \bar{\sigma} & \sigma & \bar{\omega} & \omega \\
     1 & \bar{\rho}  & \rho & \bar{\tau} & \tau
      \end{array}\right)
 \begin{array} {c} 1 \\ m_1/2 \\ m_1/2 \\ m_2/2 \\ m_2/2  \end{array}.
$$
Let $n=k_1+k_2+1$. The entries $\rho,\omega, \tau,$ and $\sigma$  are one of the three cases:
\begin{enumerate}
\item \label{t:rs20i}
 $ \rho = r_1/2, \
   \sigma=  (r_2 + \sqrt{-b})/2,\
   \tau = (t_1 + \sqrt{-z})/2, \
   \omega =  t_2/2,
$
where $  b = {nk_1}/{m_2},  z = {nk_2}/{m_1}.$
 \item  \label{t:rs20ii}
 $
 \rho =  \left( r_1 + \sqrt{-y}\right)/2, \
  \sigma = {r_2}/{2}, \
\tau = {t_1}/{2}, \
\omega =\left( t_2 + \sqrt{-c}\right)/2,
$
where
$ y = {nk_1}/{m_1},  c ={nk_2}/{m_2}.$
\item \label{t:rs20iii}
$ \rho = (r_1 + \sqrt{-y})/2, \
   \tau =  (t_1 + \sqrt{-z})/2, \
   \sigma=  (r_2 + \sqrt{-b})/2, \
   \omega =  (t_2 - \sqrt{-c})/2,
$
where all of $b, c, y,z$ are positive.
\end{enumerate}
\end{thm}

In this subsection, we refer to the three character tables in the above theorem as type 1, 2, and  3, respectively. %To prove main result, we also need the intersection matrices of skew-symmetric association schemes having character table as type I and II.

\begin{prop}\label{skew-1}
If $\mathfrak{X}$ has the character table of type $h$ for some $h\in\{1,2\}$, then the digraph $(X,R_i\cup R_{i+2})$ has $5$ distinct eigenvalues for $i\in\{1,2\}$.
\end{prop}
\begin{proof}
Let $\tilde{\mathfrak{X}}=(X,\{R_0,R_1\cup R_2,R_3\cup R_4\})$ be the symmetrization of $\mathfrak{X}$ having the character table $\tilde{P}$ as \eqref{sym character table}. Since both the graphs $(X,R_1\cup R_2)$ and $(X,R_3\cup R_4)$ are strongly regular, from Lemma \ref{row} and Proposition \ref{disconnected eigen}, we may assume that $(X,R_1\cup R_2)$ is connected. Let $(X,R_1\cup R_2)$ have parameters $(n,k,\lambda,\mu)$ with $k=k_1$. Since $(X,R_1\cup R_2)$ has eigenvalues $k,r_1,r_2$ from \eqref{sym character table}, one has $k=\mu-r_1r_2$ by \eqref{mu}. In view of \eqref{lambda}, we get $\lambda=\mu+r_1+r_2$. Since $(X,R_1\cup R_2)$ is connected, one obtains $\mu>0$. By substituting $k=\mu-r_1r_2$ and $\lambda=\mu+r_1+r_2$ into \eqref{k2}, we have
\begin{align}\label{n}
n=\frac{(\mu-r_1r_2-r_2)(\mu-r_1r_2-r_1)}{\mu}.
\end{align}
This implies
\begin{align}\label{k_2}
k_2=\frac{(r_2+1)(r_1+1)(r_1r_2-\mu)}{\mu}.
\end{align}
By substituting \eqref{n}, $k=\mu-r_1r_2$ and $\lambda=\mu+r_1+r_2$ into \eqref{mulitiplicity}, one gets
\begin{align}
m_1&=\frac{(r_2+1)(\mu-r_1r_2-r_2)(\mu-r_1r_2)}{\mu(r_2-r_1)},\label{m1}\\
m_2&=\frac{(r_1+1)(\mu-r_1r_2-r_1)(\mu-r_1r_2)}{\mu(r_1-r_2)}.\label{m2}
\end{align}

By Theorem \ref{t:rs20} \ref{t:rs20i} and \ref{t:rs20ii}, the digraph $(X,R_i\cup R_{i+2})$ has eigenvalues $(k_1+k_2)/2, \rho+\tau,\bar{\rho}+\bar{\tau},\sigma+\omega$, and $\bar{\sigma}+\bar{\omega}$ for $i\in\{1,2\}$, where $\rho+\tau=(r_1+t_1+\sqrt{-nk_l/m_1})/2$ and $\sigma+\omega=(r_2+t_2+\sqrt{-nk_h/m_2})/2$ with $\{l,h\}=\{1,2\}$. Since all of $n,k_1,k_2,m_1,m_2$ are positive, from Theorem \ref{generates}, it suffices to show that $\rho+\tau\neq\sigma+\omega$.

Assume the contrary, namely, $\rho+\tau=\sigma+\omega$. It follows that
\begin{align}
(r_1+t_1+\sqrt{-nk_l/m_1})/2=\rho+\tau=\sigma+\tau=(r_2+t_2+\sqrt{-nk_h/m_2})/2,\nonumber
\end{align}
and so $k_lm_2=k_hm_1$. By substituting $k=\mu-r_1r_2$ and \eqref{k_2}--\eqref{m2} into the latter equation, one gets
\begin{align}\label{case 1}
\frac{(\mu-r_1r_2)^2(r_l+1)(r_h^2+r_h-\mu)(\mu-r_1r_2-r_1-r_2-1)}{\mu^2(r_1-r_2)}=0.
\end{align}
By Subsection 2.2 and Proposition \ref{disconnected eigen}, we obtain $r_1r_2\leq0$ and $-1\notin\{r_1,r_2\}$. Since $k_2>0$, from \eqref{mu} and \eqref{k_2}, one has $(r_2+1)(r_1+1)<0$. Equation \eqref{case 1} implies $\mu=r_h^2+r_h$. Since $\lambda=\mu+r_1+r_2$, we get $\lambda=r_h^2+2r_h+r_l$.

In view of \cite[Theorem 2.2 (i) and (ii)]{JM11} and Lemma \ref{row}, we get $p_{1,3}^1-p_{1,3}^2=t_l/2$ and $p_{1,3}^3-p_{1,3}^4=r_h/2$, which imply that both $t_l$ and $r_h$ are even. Since $1+r_l+t_l=0$ from Lemma \ref{row}, $r_l$ is odd. Since $\lambda=r_h^2+2r_h+r_l$, from \cite[Theorem 2.2 (i) and (ii)]{JM11} again, we obtain
\begin{align}
p_{1,1}^1=\frac{\lambda+r_l}{4}=\frac{r_h(r_h+2)+2r_l}{4}.\nonumber
\end{align}
The fact that $r_h$ is even implies $2\mid r_l$, a contradiction.
\end{proof}

\begin{prop}\label{skew-2}
If $\mathfrak{X}$ has the character table of type 3, then the digraph $(X,R_i)$ has $5$ distinct eigenvalues for $1\leq i\leq4$.
\end{prop}
\begin{proof}
Let the symmetrization of $\mathfrak{X}$ have the character table $\tilde{P}$ as \eqref{sym character table}. By Theorem \ref{t:rs20} \ref{t:rs20iii}, the digraph $(X,R_i)$ with $i\in\{1,2\}$ has eigenvalues $k_1/2, (r_1\pm\sqrt{-y})/2$, and $(r_2\pm\sqrt{-b})/2$, and the digraph $(X,R_j)$ with $j\in\{3,4\}$ has eigenvalues $k_2/2, (t_1\pm\sqrt{-z})/2$, and $(t_2\pm\sqrt{-c})/2$, where all of $b,c,y,z$ are positive. By Theorem \ref{generates}, it suffices to show that $r_1\neq r_2$ and $t_1\neq t_2$.

Since $\tilde{P}$ is a change-of-basis matrix between the adjacency matrices of $\tilde{\mathfrak{X}}$ and the primitive idempotents of its Bose-Mesner algebra, $\tilde{P}$ is not singular. By Lemma \ref{row} and \eqref{sym character table}, one has $1+r_1+t_1=1+r_2+t_2=0$, which implies $r_1\neq r_2$ and $t_1\neq t_2$.
\end{proof}

\subsection{$\mathfrak{X}$ is not skew-symmetric}

Since $\mathfrak{X}$ is not skew-symmetric, we have $R_1^{\top}=R_{1}$ and $R_2^{\top}=R_2$. Let $\tilde{\mathfrak{X}}=(X,\{\tilde{R}_i\}_{i=0}^3)$ be its symmetrization with $\tilde{R}_3=R_3\cup R_4$ and $\tilde{R}_i=R_i$ for $0\leq i\leq 2$. Set $\Gamma^{(i)}=(X,\tilde{R}_i)$ for $1\leq i\leq3$. In \cite{ERD99}, van Dam completely classified symmetric $3$-class association schemes.

\begin{lemma}\label{sym 3-class}
{\rm (\cite[Section 7]{ERD99})} Suppose that $\tilde{\mathfrak{X}}$ is not amorphic. Then $\tilde{\mathfrak{X}}$ belongs to exactly one of the following categories:
\begin{enumerate}
\item\label{sym 3-class-2} $3$-class association schemes with at least one graph $\Gamma^{(i)}$ that is the disjoint union of $N>1$ ($N\in\mathbb{N}$) connected strongly regular graphs with the same parameters;

\item\label{sym 3-class-3} $3$-class association schemes with at least one graph $\Gamma^{(i)}$ having $4$ distinct eigenvalues.
\end{enumerate}
\end{lemma}

Let $E_0,E_1,E_2,E_3,E_4$ be the primitive idempotents of $\mathfrak{X}$, and $\tilde{E}_0,\tilde{E}_1,\tilde{E}_2,\tilde{E}_3$ be the primitive idempotents of $\tilde{\mathfrak{X}}$. In Lemma \ref{sym 3-class second}, the shape of the character table of the first class of graphs from Lemma \ref{sym 3-class} is described.

\begin{lemma}\label{sym 3-class second}
{\rm (\cite[Lemma 3.3]{GM})}~Suppose that $\Gamma^{(1)}$ is the disjoint union of $N>1$ $(N\in\mathbb{N}$) connected strongly regular graphs, each one with valency $k$, eigenvalues $r$ and $s$ $(r\neq k, s\neq k)$, and $w$ vertices. Then the character table $\tilde{\mathfrak{X}}$ has the form
\begin{align}\label{equ-2.1}
\tilde{P}=\left(\begin{array}{cccc}
1 & k & w-k-1 & (N-1)w \\
1 & k & w-k-1 & -w\\
1 & r & -1-r & 0\\
1 & s & -1-s & 0
\end{array}\right).
\end{align}
\end{lemma}

%\subsubsection{$\Gamma^{(i)}$ is the disjoint union of $N>1$ connected strongly-regular graphs with the same parameters for some $i$}

\begin{prop}\label{2.1}
Suppose that $\Gamma^{(i)}$ for some $i\in\{1,2\}$ is the disjoint union of $N>1$ connected strongly regular graphs, each one with valency $k$, eigenvalues $r$ and $s$ $(r\neq k, s\neq k)$, and $w$ vertices. Then $(X,R_3)$ or $(X,R_i\cup R_3)$ has $5$ distinct eigenvalues.
\end{prop}
\begin{proof}
Without loss of generality, we may assume $i=1$. In view of Lemma \ref{regular} and Theorem \ref{three eigenvalue}, we have $k>r,s$ and $r\neq s$. By Lemma \ref{sym 3-class second}, we may assume that the character table $\tilde{P}$ of $\tilde{\mathfrak{X}}$ has the form as \eqref{equ-2.1}. By replacing the roles of $\tilde{E}_2$ and $\tilde{E}_3$, from Theorem \ref{Bannai-Muzychuk}, we may assume that $\tilde{E}_1=E_1+E_2$ and $\tilde{E}_l=E_{l+1}$ for $l\in\{2,3\}$, or $\tilde{E}_3=E_3+E_4$ and $\tilde{E}_l=E_{l}$ for $l\in\{1,2\}$.

Suppose $\tilde{E_1}=E_1+E_2$ and $\tilde{E}_l=E_{l+1}$ for $l\in\{2,3\}$. By \eqref{equ-2.1} and Theorem \ref{fission scheme}, the character table of $\mathfrak{X}$ has the following form:
\begin{align}
P=\left(\begin{array}{ccccc}
1 & k & w-k-1 & (N-1)w/2 & (N-1)w/2\\
1 & k & w-k-1 & (-w+\sqrt{a})/2 & (-w-\sqrt{a})/2\\
1 & k & w-k-1 & (-w-\sqrt{a})/2 & (-w+\sqrt{a})/2\\
1 & r & -1-r & 0 & 0\\
1 & s & -1-s & 0 & 0
\end{array}\right)\nonumber
\end{align}
with $a<0$. Then the digraph $(X,R_1\cup R_3)$ has eigenvalues $k+(N-1)w/2,(2k-w\pm\sqrt{a})/2,r$, and $s$. Since $k>r,s$ and $r\neq s$, the digraph $(X,R_1\cup R_3)$ has $5$ distinct eigenvalues.

Suppose $\tilde{E}_3=E_3+E_4$ and $\tilde{E}_h=E_{h}$ for $h\in\{1,2\}$. By \eqref{equ-2.1} and Theorem \ref{fission scheme}, the character table of $\mathfrak{X}$ has the following form:
\begin{align}
P=\left(\begin{array}{ccccc}
1 & k & w-k-1 & (N-1)w/2 & (N-1)w/2\\
1 & k & w-k-1 & -w/2 & -w/2\\
1 & r & -1-r & 0 & 0\\
1 & s & -1-s & \sqrt{a} & -\sqrt{a}\\
1 & s & -1-s & -\sqrt{a} & \sqrt{a}
\end{array}\right)\nonumber
\end{align}
with $a<0$. Then the digraph $(X,R_3)$ has eigenvalues $(N-1)w/2,\pm\sqrt{a},-w/2$, and $0$. Since $N,w>1$, the digraph $(X,R_3)$ has $5$ distinct eigenvalues.
\end{proof}

\begin{prop}\label{2.2}
Suppose that $\Gamma^{(3)}$ is the disjoint union of $N>1$ connected strongly regular graphs, each one with valency $k$, eigenvalues $r$ and $s$ $(r\neq k, s\neq k)$, and $w$ vertices. Then $(X,R_i\cup R_3)$ has $5$ distinct eigenvalues for some $i\in\{1,2,3\}$.
\end{prop}
\begin{proof}
Since $(X,\tilde{R}_3)$ is the disjoint union of connected strongly regular graphs, from Theorem \ref{three eigenvalue}, $k$, $r$, and $s$ are distinct. By Lemma \ref{sym 3-class second}, there exists $i\in\{1,2\}$ such that the valency of $\Gamma^{(i)}$ is not $w-k-1$. Without loss of generality, we may assume $i=2$. By Lemma \ref{sym 3-class second}, the character table of $\tilde{\mathfrak{X}}$ has the form
\begin{align}\label{equ-2.2}
\tilde{P} =\left(\begin{array}{cccc}
1 & w-k-1 & (N-1)w & k\\
1 & w-k-1 & -w & k\\
1 & -1-r & 0 & r\\
1 & -1-s & 0 & s
\end{array}\right).
\end{align}
By replacing the roles of $\tilde{E}_2$ and $\tilde{E}_3$, from Theorem \ref{Bannai-Muzychuk}, we may assume that $\tilde{E_1}=E_1+E_2$ and $\tilde{E}_h=E_{h+1}$ for $h\in\{2,3\}$, or $\tilde{E}_3=E_3+E_4$ and $\tilde{E}_h=E_{h}$ for $h\in\{1,2\}$.

\textbf{Case 1.} $\tilde{E_1}=E_1+E_2$ and $\tilde{E}_h=E_{h+1}$ for $h\in\{2,3\}$.

By \eqref{equ-2.2} and Theorem \ref{fission scheme}, the character table of $\mathfrak{X}$ has the following form:
\begin{align}
P=\left(\begin{array}{ccccc}
1 & w-k-1 & (N-1)w & k/2 & k/2\\
1 & w-k-1 & -w & (k+\sqrt{a})/2 & (k-\sqrt{a})/2\\
1 & w-k-1 & -w & (k-\sqrt{a})/2 & (k+\sqrt{a})/2\\
1 & -1-r & 0 & r/2 & r/2\\
1 & -1-s & 0 & s/2 & s/2
\end{array}\right)\nonumber
\end{align}
with $a<0$. Then the digraph $(X,R_3)$ has eigenvalues $k/2,(k\pm\sqrt{a})/2,r/2$, and $s/2$. Since $k,r$, and $s$ are distinct, the digraph $(X,R_3)$ has $5$ distinct eigenvalues.

\textbf{Case 2.} $\tilde{E}_3=E_3+E_4$ and $\tilde{E}_h=E_{h}$ for $h\in\{1,2\}$.

By \eqref{equ-2.2} and Theorem \ref{fission scheme}, the character table of $\mathfrak{X}$ has the following form:
\begin{align}
P=\left(\begin{array}{ccccc}
1 & w-k-1 & (N-1)w & k/2 & k/2\\
1 & w-k-1 & -w & k/2 & k/2\\
1 & -1-r & 0 & r/2 & r/2\\
1 & -1-s & 0 & (s+\sqrt{a})/2 & (s-\sqrt{a})/2\\
1 & -1-s & 0 & (s-\sqrt{a})/2 & (s+\sqrt{a})/2
\end{array}\right)\nonumber
\end{align}
with $a<0$. Then the digraph $(X,R_2\cup R_3)$ has eigenvalues $k/2+(N-1)w,(s\pm\sqrt{a})/2,k/2-w$, and $r/2$. To prove that $(X,R_2\cup R_3)$ has $5$ distinct eigenvalues, it suffices to show that $w\neq(k-r)/2$ since $r<k$ from Lemma \ref{regular}.

Assume the contrary, namely, $w=(k-r)/2$. Then $r<0$ as $w>k$. Since $(X,\tilde{R}_3)$ is the disjoint union of connected strongly regular
graphs on $w$ vertices with the parameters $(w,k,\lambda,\mu)$, one obtains $\mu>0$. By \eqref{lambda} and \eqref{mu}, we have $\lambda=\mu+r+s$ and $k=\mu-rs$. In view of \eqref{k2}, one gets
\begin{align}\label{w}
w=\frac{(\mu-rs-s)(\mu-rs-r)}{\mu}.
\end{align}
By substituting
$w=(k-r)/2=(\mu-rs-r)/2$ into \eqref{w}, we obtain
\begin{align}
\frac{(\mu-rs-r)(\mu-2rs-2s)}{\mu}=0.\nonumber
\end{align}
Since $w>0$, from \eqref{w}, one has $\mu=2s(r+1)$. By Subsection 2.2, we have $k> s\geq0>r$. The fact $\mu>0$ implies $-1<r<0<s$. By Proposition \ref{integer}, $(X,\tilde{R}_3)$ is a conference graph. Then, given $m_1=m_2$, \eqref{mulitiplicity} implies $2k=(w-1)(\mu-\lambda)$. By substituting $w,k,\lambda,\mu$ into the latter equation, we get $(r+2)(rs+s^2-r+3s)=0$. Since $-1<r<0<s$, we have $r=-\frac{s(s+3)}{s-1}$, and so $s>1$. Then
$\mu=2s(r+1)=-2s(s+1)^2/(s-1)<0,$ a contradiction.
\end{proof}

\begin{prop}\label{case 3}
Suppose that $\Gamma^{(i)}$ has $4$ distinct eigenvalues for some $i\in\{1,2,3\}$. Then $(X,R_j\cup R_3)$ has $5$ distinct eigenvalues for some $j\in\{1,2,3\}$.
\end{prop}
\begin{proof}
By Theorems \ref{one pair} and \ref{generates}, we only need to consider the case that $\Gamma^{(3)}$ does not have $4$ distinct eigenvalues. Without loss of generality, we may assume that $\Gamma^{(1)}$ has $4$ distinct eigenvalues. It follows from Lemma \ref{row} that the character table of $\tilde{\mathfrak{X}}$ has the following form:
\begin{align}\label{equ-3}
\tilde{P}=\left(\begin{array}{cccc}
1 & k_1 & k_2 & 2k_3 \\
1 & u & -1-r-u & r\\
1 & v & -1-s-v & s\\
1 & w & -1-t-w & t
\end{array}\right),
\end{align}
where $k_1,u,v,w$ are distinct and $|\{2k_3,r,s,t\}|<4$. Since $(X,\tilde{R}_1)$ (resp. $(X,\tilde{R}_3)$) is a graph of valency $k_1$ with eigenvalues $k_1,u,v,w$ (resp. valency $2k_3$ with eigenvalues $2k_3,r,s,t$) from \eqref{equ-3}, we have $k_1>u,v,w$ and $2k_3\geq r,s,t$ by Lemma \ref{regular}.

By Theorem \ref{Bannai-Muzychuk}, we may assume $\tilde{E}_1=E_1+E_2$, $\tilde{E}_2=E_3$, and $\tilde{E}_3=E_4$. In view of \eqref{equ-3} and Theorem \ref{fission scheme}, the character table of $\tilde{X}$ has the following form:
\begin{align}\label{equ-3-1}
P=\left(\begin{array}{ccccc}
1 & k_1 & k_2 & k_3 & k_3 \\
1 & u & -1-r-u & (r+\sqrt{a})/2 & (r-\sqrt{a})/2\\
1 & u & -1-r-u & (r-\sqrt{a})/2 & (r+\sqrt{a})/2\\
1 & v & -1-s-v & s/2 & s/2\\
1 & w & -1-t-w & t/2 & t/2
\end{array}\right)
\end{align}
with $a<0$.

Assume the contrary, namely, none of the digraphs $(X,R_3)$, $(X,R_1\cup R_3)$, and $(X,R_2\cup R_3)$ have $5$ distinct eigenvalues. Suppose $s=t$. By \eqref{equ-3-1}, $(X,R_1\cup R_3)$ has eigenvalues $k_1+k_3,(2u+r\pm\sqrt{a})/2,v+s/2$, and $w+s/2$. The fact $w\neq v$ implies $k_1+k_3\in\{v+s/2,w+s/2\}$, contrary to the fact that $k_1>v,w$ and $2k_3\geq s$. Then $s\neq t$.

By \eqref{equ-3-1}, $(X,R_3)$ has eigenvalues $k_3,(r\pm\sqrt{a})/2,s/2$, and $t/2$. Since $(X,R_3)$ has less than $5$ distinct eigenvalues, we have $2k_3\in\{t,s\}$. Without loss of generality, we may assume $s=2k_3$. Then, by Proposition \ref{multiplicity}, $\Gamma^{(3)}$ is not connected. Since all eigenvalues of $\Gamma^{(3)}$ are $2k_3,r$, and $t$ from \eqref{equ-3}, by Lemma \ref{diameter} and Theorem \ref{three eigenvalue}, $\Gamma^{(3)}$ is the disjoint union of complete graphs with the same size, or the disjoint union of connected strongly regular graphs with the same parameters. By Proposition \ref{2.2}, $(X,\tilde{R}_3)$ is the disjoint union of $N>1$ complete graphs $K_{2k_3+1}$. Since $t\neq 2k_3$, from Theorem \ref{three eigenvalue}, we have $r\in\{2k_3,t\}$. In view of Theorem \ref{disconnected}, $\Gamma^{(3)}$
is a disconnected strongly regular graph, and Proposition \ref{disconnected eigen} implies $t=-1$.

To distinguish the notations between $\mathfrak{X}$ and $\tilde{\mathfrak{X}}$, let $p_{i,j}^{h}$ denote the intersection numbers belonging to $\mathfrak{X}$ with $0\leq i,j,h\leq4$, and $\tilde{p}_{i',j'}^{h'},\tilde{m}_{i'}$ the parameters belonging to $\tilde{\mathfrak{X}}$ with $0\leq i',j',h'\leq 3$. By \cite[Theorem 2.6 (iii)]{GLC06}, one gets $p_{3,1}^3=(\tilde{p}_{3,1}^3+u)/2$ and $p_{3,2}^3=(\tilde{p}_{3,2}^3-1-r-u)/2$. Since $(X,\tilde{R}_3)$ is the disjoint union of $N>1$ complete graphs $K_{2k_3+1}$, we have $p_{3,l}^3=\tilde{p}_{3,l}^3=0$ for $l\in\{1,2\}$. It follows that $u=0$ and $r=-1$.

By \eqref{equ-3-1}, $(X,R_1\cup R_3)$ has eigenvalues $k_1+k_3,(-1\pm\sqrt{a})/2,v+k_3$, and $w-1/2$. Since $v,w<k_1$ and $(X,R_1\cup R_3)$ has less than $5$ distinct eigenvalues, one has $w=v+k_3+1/2$, which implies \begin{align}\label{equ-3'}
\tilde{P}=\left(\begin{array}{cccc}
1 & k_1 & k_2 & 2k_3 \\
1 & 0 & 0 & -1\\
1 & v & -1-v-2k_3 & 2k_3\\
1 & 1/2+v+k_3 & -1/2-v-k_3 & -1
\end{array}\right).
\end{align}

Since $\Gamma^{(3)}$ is the disjoint union of $N>1$ complete graphs $K_{2k_3+1}$, the configuration $\hat{\mathfrak{X}}:=(X,\{\hat{R}_0,\hat{R}_1,\hat{R}_2\})$ is a symmetric $2$-class association scheme with $\hat{R}_1=\tilde{R}_1\cup\tilde{R}_2$ and $\hat{R}_2=\tilde{R}_3$. Let $\hat{E}_0,\hat{E}_1,\hat{E}_2$ be the primitive idempotents and $\hat{m}_0,\hat{m}_1,\hat{m}_2$ be multiplicities of $\hat{\mathfrak{X}}$. Since the strongly regular graph $\Gamma^{(3)}$ is not connected, from Lemma \ref{row} and Propositions \ref{multiplicity}, \ref{disconnected eigen}, the character table $\hat{\mathfrak{X}}$ has the form:
\begin{align}
\hat{P}=\left(\begin{array}{cccc}
1 & (N-1)(2k_3+1) & 2k_3 \\
1 & 0 & -1\\
1 & -1-2k_3 & 2k_3
\end{array}\right)\begin{array}{c} 1 \\ \hat{m}_1=2Nk_3 \\ \hat{m}_2=N-1  \end{array},\nonumber
\end{align}
where $k_1+k_2=(N-1)(2k_3+1)$. By Theorem \ref{Bannai-Muzychuk} and \eqref{equ-3'}, one has $\tilde{E}_1+\tilde{E}_3=\hat{E}_1$ and $\tilde{E}_2=\hat{E}_2$, which imply $\tilde{m}_1+\tilde{m}_3=\hat{m}_1=2Nk_3$ and $\tilde{m}_2=\hat{m}_2=N-1$. In view of \cite[(2.1)]{EB93}, we get $v=-(2k_3+1)k_1/(k_1+k_2)=-k_1/(N-1)$. Since the trace of the adjacency matrix of $\tilde{R}_1$ is $0$, one obtains
\begin{align}
k_1+v\tilde{m}_2+(v+k_3+1/2)\tilde{m}_3=0.\nonumber
\end{align}
Since $v=-k_1/(N-1)$ and $\tilde{m}_2=N-1$, one gets $(v+k_3+1/2)\tilde{m}_3=0$.
Since $\tilde{m}_3>0$, we have $v=-k_3-1/2$, and so $w=v+k_3+1/2=0$, contrary to the fact $w\neq u$.
\end{proof}

\begin{proof}[Proof of Theorem \ref{main}]
\ref{main1-1} Suppose that $\mathfrak{X}$ is skew-symmetric. By Theorem \ref{t:rs20} and Propositions \ref{skew-1}, \ref{skew-2}, $(X,R_{i}\cup R_3)$ has $5$ distinct eigenvalues for some $i\in\{1,3\}$. Suppose that $\mathfrak{X}$ is not skew-symmetric. It follows that $\mathfrak{X}$ has exactly one pair nonsymmetric relations. Then the symmetrization of $\mathfrak{X}$ is a $3$-class association scheme. In view of Theorem \ref{main-general}, we only need to consider the case that the symmetrization of $\mathfrak{X}$ is not amorphic. By Lemma \ref{sym 3-class} and Propositions \ref{2.1}--\ref{case 3}, $(X,R_{i}\cup R_3)$ has $5$ distinct eigenvalues for some $i\in\{1,2,3\}$. Thus, (i) is valid.

\ref{main1-2} is immediate from \ref{main1-1} and Theorem \ref{generates}.
\end{proof}

\section*{Acknowledgements}

The author would like to thank the anonymous reviewer for careful reading of
the manuscript of the paper and invaluable suggestions. Y.~Yang is supported by NSFC (12101575, 52377162).

\section*{Data Availability Statement}

No data was used for the research described in the article.

\end{document}